\documentclass[11pt,a4paper]{amsart}
\usepackage{amsmath}
\usepackage{amsfonts}
\usepackage{graphicx}
\usepackage{framed}
\usepackage{amssymb}
\usepackage{xcolor}

\usepackage[colorlinks=true,urlcolor=blue,
citecolor=blue,linkcolor=blue,linktocpage,pdfpagelabels,
bookmarksnumbered,bookmarksopen]{hyperref}
\usepackage{hyperref}

\usepackage{etex}
\usepackage{latexsym}
\usepackage[english]{babel}
\usepackage[utf8x]{inputenc}

\def \X {\mathbb{X}(\Omega)}

\def \N {\mathbb{N}}
\def \R {\mathbb{R}}

\def \de {\partial}

\def \LL {\mathcal{L}_{\alpha}}

\usepackage{amsthm}
\theoremstyle{definition}
\newtheorem{definition}{Definition}[section]

\newtheorem{remark}[definition]{Remark}

\theoremstyle{plain}
\newtheorem{theorem}[definition]{Theorem}
\newtheorem{proposition}[definition]{Proposition}
\newtheorem{lemma}[definition]{Lemma}

\makeatletter
\@namedef{subjclassname@2020}{\textup{2020} Mathematics Subject Classification}
\makeatother

\numberwithin{equation}{section}
\begin{document}
\title[Nonpositive mixed operators]{Variational methods for nonpositive mixed local--nonlocal operators}


\author[A.\,Maione]{Alberto Maione$^{\,1^*}$}
\author[D.\,Mugnai]{Dimitri Mugnai$^{\,2\,\dagger}$}
\author[E.\,Vecchi]{Eugenio Vecchi$^{\,3\,\dagger}$}
\address
{$^1$ Abteilung f\"{u}r Angewandte Mathematik\newline\indent Albert-Ludwigs-Universit\"{a}t Freiburg \newline\indent Hermann-Herder-Strasse 10, 79104 Freiburg im Breisgau, Germany}
\address
{$^2$ Dipartimento di Scienze Ecologiche e Biologiche (DEB)\newline\indent Università degli Studi della Tuscia \newline\indent
Largo dell'Università, 01100 Viterbo, Italy}
\address
{$^3$ Dipartimento di Matematica\newline\indent ALMA MATER STUDIORUM - Università di Bologna \newline\indent\bigskip 
Piazza di Porta San Donato 5, 40126 Bologna, Italy
\newline $^*$ Corresponding author\newline Correspondence to \url{alberto.maione@mathematik.uni-freiburg.de}\newline $^\dagger$ Contributing authors: these authors contributed equally to this work.}
%



\begin{abstract}
We prove the existence of a weak solution for boundary value problems driven by a mixed local--nonlocal operator. The main novelty is that such an operator is allowed to be nonpositive definite.
\end{abstract}
\maketitle 
\makeatletter
\enddoc@text
\let\enddoc@text\empty 
\makeatother
\section{Introduction}\label{sec.Intro}
In this paper we are concerned with semilinear elliptic problems driven by a mixed local--nonlocal operator of the form
\[
\LL u := -\Delta u +\alpha (-\Delta)^s u\,.
\]
Here $\alpha \in \mathbb{R}$ with no a priori restrictions, $\Delta u$ denotes the classical Laplace operator while $(-\Delta)^s u$, for fixed $s\in (0,1)$ is the fractional Laplacian, usually defined as
\[
(-\Delta)^s u(x) := C(n,s) \,\mbox{P.V.}\int_{\R^n}\frac{u(x)-u(y)}{|x-y|^{n+2s}}\, dy\,,
\]
where P.V. denotes the Cauchy principal value, that is
\[
\mbox{P.V.}\int_{\R^n}
  \frac{u(x)-u(y)}{|x-y|^{n+2s}}\, dy\\
  =\lim_{\epsilon\to0}\int_{\{y\in \R^n\,:\,|y-x|\geq \epsilon\}}
  \frac{u(x)-u(y)}{|x-y|^{n+2s}}\, dy\,.
\]
Clearly, when $\alpha=0$, one recovers the classical Laplacian, while, for $\alpha >0$, one is led to consider a positive operator which can be considered as a particular instance of an infinitesimal generator of a stochastic process involving a Brownian motion and a pure jump L\'{e}vy process. 

We briefly recall, focusing merely on the more recent (elliptic) PDEs oriented literature, that problems driven by operators of mixed type, even with a nonsingular nonlocal operator \cite{DPFR}, have raised a certain interest in the last few years, for example in connection with the study of optimal animal foraging strategies (see e.g. \cite{DV} and the references therein). From the pure mathematical point of view, the superposition of such operators generates a lack of scale invariance which may lead to unexpected complications.

At the present stage, and without aim of completeness, the investigations have taken into consideration interior regularity and maximum principles (see e.g. \cite{BDVV, BiMo,CDV22,DeFMin,GarainKinnunen,GarainLindgren}), boundary Harnack principle \cite{CKSV}, boundary regularity and overdetermined problems \cite{BMS, SUPR}, qualitative properties of solutions \cite{BDVVAcc}, existence of solutions and asymptotics (see e.g. \cite{BDVV5, BMV, BMV2, BSM, SilvaSalort, DPLV1, DPLV2, GarainUkhlov, MPL, SalortVecchi}) and shape optimization problems \cite{BDVV2, BDVV3, GoelSreenadh}.
\medskip

In this paper we deal with the following boundary value problem:
\begin{equation}\label{eq:Problem}
\left\{ \begin{array}{rl}
           \LL u = f(x,u)\,, & \textrm{in } \Omega\\
           u = 0\,, & \textrm{in } \mathbb{R}^{n}\setminus \Omega
\end{array},\right.
\end{equation}
where $\Omega \subset \mathbb{R}^n$, with $n>2$, is an open and bounded set with $C^1$-smooth boundary. 

In particular, we consider two different sets of assumptions on the nonlinear term $f:\Omega \times \mathbb{R} \to \mathbb{R}$, which is always supposed to be a Carath\'{e}odory function: first, we deal with the {\it asymptotically linear case}, and second with the {\it superlinear and subcritical case}.

In the first setting, we assume that $f$ has at most linear growth, according to the following assumption: there exist a function $a\in L^{2}(\Omega)$ and $b\in \mathbb{R}$ such that
\begin{equation}\label{f_{lg}}
    |f(x,t)| \leq a(x) + b|t|\quad \textrm{for all } t \in \mathbb{R} \textrm{ and  for a.e. } x \in \Omega\,.
\end{equation}
In order to state our result, we also need the following measurable functions:
\begin{equation}\label{eq:Defvv}
\underline{v}(x) := \liminf_{|t|\to +\infty}\dfrac{f(x,t)}{t} \leq \limsup_{|t|\to +\infty}\dfrac{f(x,t)}{t}=: \overline{v}(x).
\end{equation}
A trivial example of such function $f$ is given by $f(x,t)=\lambda t +a(x)$, where $\lambda\in \R$ and $a\in L^2(\Omega)$.

Instead, in the superlinear and subcritical case, we assume:
\begin{equation}\tag{$i$}\label{(i)}
    f(x,0)=0\quad\textrm{for a.e. } x \in \Omega\,;
\end{equation}
there exist a function $a\in L^{\infty}(\Omega)_+$, 
a number $b\in \mathbb{R}$ and $r\in(2,2^*)$ such that
\begin{equation}\tag{$ii$}\label{(ii)}
    |f(x,t)| \leq a(x) + b|t|^{r-1}\quad \textrm{for all } t \in \mathbb{R} \textrm{ and  for a.e. } x \in \Omega\,;
\end{equation}
there exist $\mu>2,\,\tilde \mu>2$, $R>0$, $c>0$, $A\in L^\infty(\Omega)$ and $d\in L^{1}(\Omega)_+$ such that 
\begin{equation}\tag{$iii$}\label{(iii)}
\begin{split}
    0<\mu F(x,t)-\mu A(x)\frac{t^2}{2}\leq f(x,t)t -A(x)t^2 \quad &\textrm{for all } |t|\geq R \textrm{ and  for a.e. } x \in \Omega\,,\\
    F(x,t)\geq c|t|^{\tilde \mu}-d(x) \quad &\textrm{for all } t \in \mathbb{R} \textrm{ and  for a.e. } x \in \Omega\,,
\end{split}
\end{equation}
where $F(x,t)=\int_0^t f(x,\sigma)\,d\sigma$ for any $t\in\mathbb{R}$.

Here $2^*$ denotes the classical Sobolev critical exponent, namely
\[
2^*=\dfrac{2n}{n-2}\,.
\]
A trivial example for $f$ in this case is given by $f(x,t)=\lambda t+|t|^{p-2}t$, with $\lambda\in \R$ and $p\in (2,2^*)$.
\begin{remark}
We remark that condition $(iii)$ means that $f(x,t)-A(x)t$ satisfies the Ambrosetti-Rabinowitz condition \cite{AR}, where the growth from below  is necessary, see \cite{addendum}. For the model just above we clearly have $A=\lambda$.
\end{remark}
As for the asymptotically linear case, we need to introduce asymptotic functions, which are now relevant as $t\to 0$:
\begin{equation}\label{iv}
    \underline{\Theta}(x) := \liminf_{t\to 0}\dfrac{f(x,t)}{t} \leq \limsup_{t\to 0}\dfrac{f(x,t)}{t}=: \overline{\Theta}(x)\,.
\end{equation}
Our aim is to prove the existence of weak solutions to problem \eqref{eq:Problem} (see section \ref{sec.NotPrel} for the precise definition).
As partially expected, if $\alpha \geq 0$, our existence results (see Theorem \ref{thmASY} and Theorem \ref{thmSL}) are either known or applications of standard variational methods (for instance, see \cite{limit} or \cite{Mugnai} when $\alpha =0$, or \cite{LIBRO} for the pure fractional case). 
A similar behaviour happens to hold if we take
\begin{equation*}
-\dfrac{1}{C}<\alpha<0\,,
\end{equation*}
where $C>0$ is the constant of the continuous embedding $H^{1}_{0} \subset H^{s}$ (see e.g. \cite{DRV}), i.e.
\begin{equation*}
[u]_s^2 := \iint_{\mathbb{R}^{2n}}\dfrac{|u(x)-u(y)|^2}{|x-y|^{n+2s}}\, dx dy \leq C \, \int_{\Omega}|\nabla u|^2 \, dx\,.
\end{equation*}
In this perspective, the probably more interesting case is for $\alpha\leq-\tfrac{1}{C}$.

Indeed, the situation becomes suddenly more delicate, mainly because the local--nonlocal operator is not more positive definite. As a consequence, the bilinear form naturally associated to it does not induce a scalar product nor a norm, the variational spectrum may exhibit negative eigenvalues and even the maximum principles may fail, see e.g. \cite{BDVV}. 
\medskip

Let $\{\lambda_k\}_{k\in \mathbb{N}}$ be the sequence of eigenvalues of $\mathcal{L}_\alpha$, see Proposition \ref{prop.Eigenvalue} for details.
The main result in the asymptotically linear case states as follows.
\begin{theorem}\label{thmASY}
Assume that $f$ satisfies \eqref{f_{lg}} and that the limits in \eqref{eq:Defvv} are uniform in $x$.
If either
\begin{itemize}
\item $\overline{v}(x) < \lambda_{1}$ for a.e. $x \in \Omega$, or
\item there exists $k \in \mathbb{N}$ such that $\lambda_{k}<\underline{v}(x) \leq \overline{v}(x)<\lambda_{k+1}$ for a.e. $x \in \Omega$,
\end{itemize}
then problem \eqref{eq:Problem} admits a weak solution.
\end{theorem}
Its counterpart in the superlinear and subcritical case states instead as follows.
\begin{theorem}\label{thmSL}
Assume $f$ satisfies \eqref{(i)}, \eqref{(ii)} and \eqref{(iii)} and that the limits in \eqref{iv} are uniform in $x$.
If either
\begin{itemize}
\item $\overline{\Theta}(x)<\lambda_1$ a.e. $x \in \Omega$ or
\item there exists $k \in \mathbb{N}$ such that $\lambda_{k}\leq\underline{\Theta}(x)\leq\overline{\Theta}(x)<\lambda_{k+1}$ for a.e. $x \in \Omega$ and
\begin{equation}\label{lastassumption}
    F(x,t)\geq \lambda_k \tfrac{t^2}{2}\quad\text{for a.e. }x \in \Omega\text{ and for any }t\in\mathbb{R}\,,
\end{equation}
\end{itemize}
then problem \eqref{eq:Problem} admits a weak solution.
\end{theorem}
We stress that, despite being slightly non-standard, the latter assumption \eqref{lastassumption} is satisfied in the model case previously mentioned.
\medskip

The paper is organized as follows: in section \ref{sec.NotPrel} we collect all the assumptions and preliminary results, including a description of the variational spectrum  (Proposition \ref{prop.Eigenvalue}). In section \ref{sec.AsyLinear} we prove Theorem \ref{thmASY}, while Theorem \ref{thmSL} is proved in section \ref{sec.Superlinear}.
\medskip

\section{Assumptions, notations and preliminary results} \label{sec.NotPrel}
Let $\Omega\subseteq\R^n$, $n>2$, be a connected and bounded open set with $C^1$-smooth boundary
 $\de\Omega$.
We define the space of solutions of problem \eqref{eq:Problem} as
\begin{equation*}
\mathbb{X}(\Omega) := \big\{u\in H^{1}(\R^n):\,
\text{$u\equiv 0$ a.e.\,on $\R^n\setminus\Omega$}\big\}.
\end{equation*}
Thanks to the regularity assumption on $\de\Omega$ (see \cite[Proposition 9.18]{Brezis}), we can identify the space $\mathbb{X}(\Omega)$ with the space $H_0^{1}(\Omega)$ in the following sense:
\begin{equation} \label{eq.identifXWzero}
    u \in H_0^{1}(\Omega)\,\,\Longleftrightarrow\,\,
    u\cdot\mathbf{1}_{\Omega}\in\mathbb{X}(\Omega)\,,
\end{equation}
where $\mathbf{1}_\Omega$ is the indicator function of $\Omega$.
From now on, we shall always identify a function $u\in H_0^{1}(\Omega)$ with $\hat{u} := u\cdot\mathbf{1}_\Omega\in\mathbb{X}(\Omega)$.

By the Poincar\'e inequality and \eqref{eq.identifXWzero}, we get that the quantity
\[
\|u\|_{\mathbb{X}} :=\left( \int_{\Omega}|\nabla u|^2\, dx\right)^{1/2},\quad u\in\mathbb{X}(\Omega)\,,\]
endows $\mathbb{X}(\Omega)$ with a structure of (real) Banach space, which is isometric to $H_0^{1}(\Omega)$.
To fix the notation, we denote by $\langle \cdot,\cdot \rangle_{\mathbb{X}}$ the scalar product which induces the above norm on $\mathbb{X}(\Omega)$.
We briefly recall that the space $\mathbb{X}(\Omega)$ is separable and reflexive, $C_0^\infty(\Omega)$ is dense in $\mathbb{X}(\Omega)$ and eventually that $\mathbb{X}(\Omega)$ compactly embeds in $L^{2}(\Omega)$ and in 
\[
H^s_0(\Omega):=\left\{H^{s}(\R^n):\,\text{$u\equiv 0$ a.e.\,on $\R^n\setminus\Omega$}\right\}
\]
by \cite[Theorem 16.1]{lions}.
\begin{definition}\label{eq:DefWeakSol}
A function $u \in \mathbb{X}(\Omega)$ is called a weak solution of \eqref{eq:Problem} if
\begin{equation*}
    \int_{\Omega}\langle \nabla u, \nabla \varphi\rangle\, dx +\alpha \iint_{\mathbb{R}^{2n}}\dfrac{(u(x)-u(y))(\varphi(x)-\varphi(y))}{|x-y|^{n+2s}}\, dxdy = \int_{\Omega}f(x,u)\varphi \, dx
\end{equation*}
for every $\varphi \in \mathbb{X}(\Omega)$.
\end{definition}
As usual, weak solutions of \eqref{eq:Problem} can be found as critical points of the functional $\mathcal{J}:\mathbb{X}(\Omega) \to \mathbb{R}$ defined as
\begin{equation*}
\mathcal{J}(u) := \dfrac{1}{2}\int_{\Omega}|\nabla u|^2 \, dx + \dfrac{\alpha}{2}\iint_{\mathbb{R}^{2n}} \dfrac{|u(x)-u(y)|^2}{|x-y|^{n+2s}}\, dxdy - \int_{\Omega}F(x,u)\, dx\,.
\end{equation*}
Here
\begin{equation*}
F(x,t) := \int_{0}^{t}f(x,\sigma)\, d\sigma, \quad t \in \mathbb{R}.
\end{equation*}
The functional $\mathcal{J}$ is Fr\'{e}chet differentiable and
\begin{align*}
\mathcal{J}'(u)(\varphi) &= \int_{\Omega}\langle \nabla u, \nabla \varphi\rangle \, dx + \alpha \iint_{\mathbb{R}^{2n}}\dfrac{(u(x)-u(y))(\varphi(x)-\varphi(y))}{|x-y|^{n+2s}}\, dxdy \\
&\quad- \int_{\Omega}f(x,u(x))\varphi(x)\, dx \quad \textrm{for every }\varphi \in \mathbb{X}(\Omega)\,.
\end{align*}
\begin{definition}
Consider the bilinear form 
$\mathcal{B}_{\alpha}:\mathbb{X}(\Omega)\times \mathbb{X}(\Omega) \to \mathbb{R}$, defined by
\begin{equation*}
\mathcal{B}_{\alpha}(u,v):= \int_{\Omega}\langle \nabla u, \nabla v\rangle\, dx +\alpha \iint_{\mathbb{R}^{2n}}\dfrac{(u(x)-u(y))(v(x)-v(y))}{|x-y|^{n+2s}}\, dxdy
\end{equation*}
for any $u,v \in \mathbb{X}(\Omega)$. We say that $u$ and $v$ are $\mathcal{B}$-orthogonal if
\begin{equation*}
\mathcal{B}_{\alpha}(u,v)=0.
\end{equation*}
\end{definition}
The terminology adopted above is justified by the fact that, for $\alpha>0$, the bilinear form $\mathcal{B}_{\alpha}$ becomes a true scalar product.

We conclude this section dealing with the eigenvalue problem associated to the operator $\LL$, that is the following boundary value problem
\begin{equation}\label{eq.EigenvalueProblem}
\left\{ \begin{array}{rl}
  \LL u = \lambda u\,, & \textrm{in } \Omega\\
  u= 0\,, & \textrm{in } \mathbb{R}^n \setminus \Omega
\end{array}\right.
\end{equation}
where $\lambda \in \mathbb{R}$.
According to Definition \ref{eq:DefWeakSol}, we give the following definition.
\begin{definition}
A number $\lambda \in \mathbb{R}$ is called a (variational) eigenvalue of $\LL$ if there exists a weak solution $u\in \mathbb{X}(\Omega)$ of \eqref{eq.EigenvalueProblem} or, equivalently, if
\begin{equation*}
\int_{\Omega}\langle \nabla u, \nabla \varphi\rangle\, dx +\alpha \iint_{\mathbb{R}^{2n}}\dfrac{(u(x)-u(y))(\varphi(x)-\varphi(y))}{|x-y|^{n+2s}}\, dxdy =\lambda \int_{\Omega}u\varphi \, dx
\end{equation*}
for every $\varphi \in \mathbb{X}(\Omega)$.
If such function $u$ exists, we call it eigenfunction corresponding to the eigenvalue $\lambda$.
\end{definition}

The next result permits a complete description of the eigenvalues and related eigenfunctions of $\LL$.

\begin{proposition}\label{prop.Eigenvalue}
The following statements hold true:
\begin{itemize}
\item[(a)] $\LL$ admits a divergent, but bounded from below, sequence of eigenvalues $\{\lambda_k\}_{k\in \mathbb{N}}$, i.e., there exists $C>0$ such that
\[
-C < \lambda_1 \leq \lambda_2  \leq \ldots \leq \lambda_k \to +\infty\,, \quad  \textrm{as } k \to +\infty.
\]
Moreover, for every $k\in \mathbb{N}$, $\lambda_{k}$ can be characterized as
\begin{equation}\label{eq:lambdaChar}
\lambda_{k} = \min_{\stackrel{u \in \mathbb{P}_{k}}{\|u\|_{L^{2}(\Omega)}=1}} \left\{\int_{\Omega}|\nabla u|^2 \, dx +\alpha \iint_{\mathbb{R}^{2n}}\dfrac{|u(x)-u(y)|^2}{|x-y|^{n+2s}}\, dxdy \right\},
\end{equation}
where 
\begin{equation*}
\mathbb{P}_{1} := \mathbb{X}(\Omega),
\end{equation*}
and, for every $k \geq 2$,
\begin{equation*}
\mathbb{P}_{k} := \left\{ u \in \mathbb{X}(\Omega): \mathcal{B}_{\alpha}(u, u_j)=0 \, \textrm{ for every } j = 1, \ldots,k-1\right\}; 
\end{equation*}
\item[(b)] for every $k\in \mathbb{N}$ there exists an eigenfunction $u_{k} \in \mathbb{X}(\Omega)$ corresponding to $\lambda_{k}$, which realizes the minimum in \eqref{eq:lambdaChar};
\item[(c)] the sequence $\{u_k\}_{k\in \mathbb{N}}$ of eigenfunctions constitutes an orthonormal basis of $L^{2}(\Omega)$. Moreover, the eigenfunctions are $\mathcal{B}$-orthogonal;
\item[(d)] for every $k\in \mathbb{N}$, $\lambda_k$ has finite multiplicity.
\end{itemize}
\begin{proof}
If $\alpha=0$ the result is the classical spectral theorem for the Laplace operator (see e.g. \cite{Brezis}). As already mentioned in the Introduction, the case $\alpha >-\tfrac{1}{C}$ (with $C>0$ being the constant of the continuous embedding $H^1_0 \subset H^s$) is also pretty standard. Therefore, in what follows we assume $\alpha \leq -\tfrac{1}{C}$.

By {\cite[Theorem 1]{interpolation}}, there exists a positive constant $C$ such that
\begin{equation*}
        [u]^2_s\leq C\|u\|^{2(1-s)}_{L^2(\Omega)}\|u\|^{2s}_{H^1(\Omega)}
    \end{equation*}
    for any $u\in\mathbb{X}(\Omega)$. We remind that the optimal constant $C$ can be explicitly computed with the help of Fourier transform (see e.g. \cite{CCMV}).
    
We now combine the previous interpolation estimate with the Young inequality (with exponents $\frac{1}{s}$ and $\frac{1}{1-s}$), i.e. for any $\varepsilon\in\mathbb{R}^+$ there exist positive constants $c_{\varepsilon},c_1,c_2$, depending on $s$ and $\varepsilon$, such that
    \begin{equation*}
    \begin{split}
        |\alpha|[u]^2_s&\leq |\alpha|C\left((1-s)c_\varepsilon\|u\|^2_{L^2(\Omega)}+s\varepsilon\|u\|^2_{H^1(\Omega)}\right)\\
        &=|\alpha|c_1\varepsilon\|u\|^2_{\mathbb{X}(\Omega)}+|\alpha|c_2\|u\|^2_{L^2(\Omega)}
    \end{split}
    \end{equation*}
    for any $u\in\mathbb{X}(\Omega)$.
    Therefore,  by choosing $\varepsilon=\frac{1}{2c_1|\alpha|}$, we get
    \begin{equation*}
        |\alpha|[u]^2_s\leq\frac{1}{2}\|u\|^2_{\mathbb{X}(\Omega)}+\gamma\|u\|^2_{L^2(\Omega)}
    \end{equation*}
    for any $u\in\mathbb{X}(\Omega)$, where $\gamma$ only depends on $s,\varepsilon$ and $\alpha$, and so
    \begin{equation}\label{controlBalpha}
    \mathcal{B}_\alpha(u,u)+\gamma\|u\|^2_{L^2(\Omega)}\geq\frac{1}{2}\|u\|^2_{\mathbb{X}(\Omega)}\quad\text{for any }u\in\mathbb{X}(\Omega)\,.
    \end{equation}

At this point, in a standard fashion (for instance, see \cite[Chapter 6]{Evans}) one can prove the existence of  an increasing sequence of eigenvalues $\{\lambda_k\}_k$  of $\LL$, with $\lambda_k\to\infty$ as $k\to\infty$,  such that every $\lambda_k$ has finite multiplicity and is given by the variational characterization in \eqref{eq:lambdaChar}. Moreover, if $e_k$ is the eigenfunction associated to $\lambda_k$, we have that  $\{e_k\}_k$ is an orthonormal basis of $L^{2}(\Omega)$ with 
\begin{equation*}
\mathcal{B}_\alpha(e_k,e_j)=\lambda_k(e_k,e_j)_{L^{2}(\Omega)}=0 \quad\text{for every }j\neq k\,,
\end{equation*}
that is, $e_k$ and $e_j$ are also $\mathcal{B}_\alpha$-orthogonal.

\end{proof}
\end{proposition}
By \cite[Theorem 5.2.4]{Frank} and Proposition \ref{prop.Eigenvalue} (a), we can also infer the existence of a positive integer $N_0 \in \mathbb{N}$ such that $\lambda_{N_0}$ is the first (not necessarily simple) positive eigenvalue. Of course, $\lambda_k >0$ for every $k > N_0$.\\
We further notice that
\begin{equation}\label{eq:Stima1}
\lambda_{k+1} \int_{\Omega}u^2 \, dx \leq \int_{\Omega}|\nabla u|^2 \, dx +\alpha \iint_{\mathbb{R}^{2n}}\dfrac{|u(x)-u(y)|^2}{|x-y|^{n+2s}}\, dxdy
\end{equation}
for every $u \in \textrm{span}(u_1, \ldots, u_k)^{\perp} = \mathbb{P}_{k+1}$ and
\begin{equation}\label{eq:Stima2}
\int_{\Omega}|\nabla u|^2 \, dx +\alpha \iint_{\mathbb{R}^{2n}}\dfrac{|u(x)-u(y)|^2}{|x-y|^{n+2s}}\, dxdy \leq \lambda_k \int_{\Omega}u^2 \, dx
\end{equation}
for every $u \in \textrm{span}(u_1, \ldots, u_k)=: H_k$.

While \eqref{eq:Stima1} directly follows from the variational characterization \eqref{eq:lambdaChar}, the latter \eqref{eq:Stima2} can be proved as follows: by assumption, let $u = \sum_{i=1}^{k}c_i u_i$. Then,
\begin{equation*}
\begin{aligned}
    \int_{\Omega}|\nabla u|^2 \, dx +\alpha \iint_{\mathbb{R}^{2n}}\dfrac{|u(x)-u(y)|^2}{|x-y|^{n+2s}}\, dxdy &= \mathcal{B}_{\alpha}(u,u) =\mathcal{B}_{\alpha}(\sum_{i=1}^{k}c_i u_i, \sum_{j=1}^{k}c_j u_j) \\
    &= \sum_{i,j=1}^{k}c_i c_j \mathcal{B}_{\alpha}(u_i,u_j) =\sum_{i=1}^{k}c_i^2 \mathcal{B}_{\alpha}(u_i,u_i) \\
    &= \sum_{i=1}^{k}c_i^2 \lambda_i \int_{\Omega}u_i^2 \, dx \leq \lambda_k \int_{\Omega}u^2 \, dx.
\end{aligned}
\end{equation*}


\section{The asymptotically linear case}\label{sec.AsyLinear}
In this section we prove Theorem \ref{thmASY}.
\begin{proof}[Proof of Theorem \ref{thmASY}]
{\bf Case 1}: $\overline{v}(x)<\lambda_{1}$ for a.e. $x \in \Omega$.
\medskip

We claim that 
\begin{equation}\label{eq.ClaimWeierstrass}
\liminf_{\|u\|_{\mathbb{X}}\to +\infty}\dfrac{\mathcal{J}(u)}{\|u\|^2_{\mathbb{X}}} >0\,.
\end{equation}
Once this is established, we have that functional $\mathcal{J}$ is coercive and sequentially weakly lower semicontinuous in $\X$, since
\begin{itemize}
\item the map $u \mapsto \|u\|_{\mathbb{X}}$ is sequentially weakly l.s.c. in $\mathbb{X}(\Omega)$ (being a norm);
\item the map $u \mapsto [u]_s$ is continuous, because $\X$ compactly embeds in $H^s_0(\Omega)$;
\item the map $u\mapsto \int_{\Omega}F(x,u)\, dx$ is continuous by $\eqref{f_{lg}}$.
\end{itemize}

By the very definition of $\overline{v}$, for every $\varepsilon>0$ there exists $R>0$ such that
\begin{equation}\label{eq.stimadav}
f(x,\sigma) < (\overline{v}(x) + \varepsilon) \sigma \quad \textrm{for every } |\sigma|>R\,.
\end{equation}
On the other hand, assumption $\eqref{f_{lg}}$ readily gives that 
\begin{equation}\label{eq.stimadaflg}
f(x,\sigma) \leq a(x) + bR \quad \textrm{for every } |\sigma|\leq R\,.
\end{equation}
Combining \eqref{eq.stimadav} with \eqref{eq.stimadaflg}, we then get that
\[
\limsup_{|\sigma|\to +\infty} \dfrac{F(x,\sigma)}{\sigma^2} \leq \dfrac{\overline{v}(x)+\varepsilon}{2} \quad \textrm{for every } \varepsilon>0
\]
and, passing to the limit as $\varepsilon \to 0^+$, we finally get that
\begin{equation}\label{eq.stimaLimSup}
\limsup_{|\sigma|\to +\infty} \dfrac{F(x,\sigma)}{\sigma^2} \leq \dfrac{\overline{v}(x)}{2}\,.
\end{equation}
Let us now proceed with the proof of \eqref{eq.ClaimWeierstrass}.
We take a sequence $\{u_j\}_j \subset \mathbb{X}(\Omega)$ such that
$\|u_j\|_{\mathbb{X}} \to +\infty$ as $j \to +\infty$ and we define the normalized sequence $w_j := \tfrac{u_j}{\|u_j\|_{\mathbb{X}}}$. Then, possibly passing to a subsequence, we can assume the existence of a function $u_0 \in \mathbb{X}(\Omega)$ such that $w_j \to u_0$ weakly in $\mathbb{X}(\Omega)$, strongly in $L^{2}(\Omega)$ and a.e. in $\Omega$. Moreover, $\|u_0\|_{\mathbb{X}}\leq 1$ and, by $\eqref{f_{lg}}$, it holds that
\[
\dfrac{F(x,u_j)}{\|u_j\|^{2}_{\mathbb{X}}} \leq \dfrac{a(x)|u_j| + b \tfrac{|u_j|^2}{2}}{\|u_j\|^{2}_{\mathbb{X}}} \to \dfrac{b}{2}u_{0}^{2}(x) \quad \textrm{in } L^{1}(\Omega)\,. 
\]
Now, write
\[
\Omega = \left\{ x\in \Omega: u_j(x) \textrm{ is bounded }\right\} \cup \left\{ x\in \Omega: |u_j(x)|\to +\infty \right\}=: \Omega_1 \cap \Omega_2\,.
\]
Of course, if $x \in \Omega_1$, then 
\begin{equation}\label{eq.inOmega1}
\lim_{j\to +\infty}\dfrac{F(x,u_j)}{\|u_j\|^2_{\mathbb{X}}} = 0
\end{equation}
while, if $x\in \Omega_2$, thanks to \eqref{eq.stimaLimSup}, we get
\begin{equation}\label{eq.inOmega2}
\limsup_{j\to +\infty}\dfrac{F(x,u_j)}{\|u_j\|^2_{\mathbb{X}}} = \limsup_{j\to +\infty} \dfrac{F(x,u_j)}{u_j^{2}(x)} \, \dfrac{u_j^{2}(x)}{\|u_j\|_{\mathbb{X}}^2} \leq \dfrac{\overline{v}(x)}{2}u_0^{2}(x)\,.
\end{equation}
By the generalized Fatou Lemma, combined with \eqref{eq.inOmega1} and \eqref{eq.inOmega2}, we have that
\begin{equation}\label{eq:StimaFconv}
\limsup_{j\to +\infty}\int_{\Omega}\dfrac{F(x,u_j)}{\|u_j\|^2_{\mathbb{X}}}\, dx \leq \int_{\Omega}\limsup_{j\to +\infty} \dfrac{F(x,u_j)}{\|u_j\|^2_{\mathbb{X}}}\, dx \leq \int_{\Omega}\dfrac{\overline{v}(x)}{2}u_0^{2}(x)\, dx\,. 
\end{equation}
On the other hand, by the definition of $\lambda_1$, we have
\begin{equation}\label{eq:StimaConk=0}
\begin{aligned}
\dfrac{1}{\|u_j\|_{\mathbb{X}}^2} &\left( \dfrac{1}{2}\int_{\Omega}|\nabla u_j|^2 \, dx + \dfrac{\alpha}{2}\iint_{\mathbb{R}^{2n}}\dfrac{|u_j(x)-u_j(y)|^2}{|x-y|^{n+2s}}\, dxdy \right) \\
&\geq \dfrac{\lambda_1}{2}\int_{\Omega}\dfrac{u_j^2(x)}{\|u_j\|^2_{\mathbb{X}}}\, dx \to \dfrac{\lambda_1}{2}\int_{\Omega}u_0^2(x) \, dx\,, \quad \textrm{as $j \to +\infty$}.
\end{aligned}
\end{equation}
To prove the validity of \eqref{eq.ClaimWeierstrass}, we have to consider two possible situations: either $u_0 \neq 0$ or $u_0 = 0$ a.e. in $\Omega$. If the first happens to be true, then we combine \eqref{eq:StimaFconv} and \eqref{eq:StimaConk=0} with the standing assumption $\overline{v}<\lambda_1$ (a.e. in $\Omega$), getting that
\[
\liminf_{j\to +\infty}\dfrac{\mathcal{J}(u_j)}{\|u_j\|^2_{\mathbb{X}}} \geq \int_{\Omega}\left(\dfrac{\lambda_1}{2}  -\dfrac{\overline{v}(x)}{2}\right)u_0^2(x)\,dx >0\,.
\]
On the contrary, if $u_0 = 0$ a.e. in $\Omega$, we notice that, by compactness,
\[
\dfrac{1}{\|u_j\|^2_{\mathbb{X}}} \iint_{\mathbb{R}^{2n}}\dfrac{|u_j(x)-u_j(y)|^2}{|x-y|^{n+2s}}\, dxdy \to 0\,, \quad \textrm{as }j \to +\infty\,,
\]
and therefore we have that
\[
\liminf_{j\to +\infty}\dfrac{\mathcal{J}(u_j)}{\|u_j\|^2_{\mathbb{X}}} \geq \dfrac{1}{2}\,.
\]
In any case, \eqref{eq.ClaimWeierstrass} holds and an application of the Weierstrass Theorem closes the proof of Case 1.
\medskip

{\bf Case 2}: there exists $k \in \mathbb{N}$ such that $\lambda_k < \underline{v}(x) \leq \overline{v}(x) < \lambda_{k+1}$ for a.e. $x \in \Omega$.
\medskip

Note that, reasoning as in the proof of \eqref{eq.stimaLimSup} in Case 1, we have that
\begin{equation}\label{eq:stimaLimInf}
\liminf_{|\sigma|\to +\infty} \dfrac{F(x,\sigma)}{\sigma^2} \geq \dfrac{\underline{v}(x)}{2}\, \quad \textrm{for a.e. }x \in \Omega\,.
\end{equation}
Take a sequence $\{u_j\}_j \subset H_k$ (recall its definition in \eqref{eq:Stima2}) such that $\|u_j\|_{\mathbb{X}} \to +\infty$ as $j \to +\infty$. Since $H_k$ is finite dimensional, we can infer the existence of a function $u_0 \in H_k$, with $\|u_0\|_{\mathbb{X}} =1$, and such that the normalized sequence $\tfrac{u_j}{\|u_j\|_{\mathbb{X}}} \to u_0$ strongly in $\mathbb{X}(\Omega)$, strongly in $L^2(\Omega)$ and a.e. in $\Omega$.
Exploiting now \eqref{eq:Stima2}, combined with the standing assumption on $\lambda_k$ and \eqref{eq:stimaLimInf}, we get
\[
\limsup_{j \to +\infty}\dfrac{\mathcal{J}(u_j)}{\|u_j\|^2_{\mathbb{X}}} \leq \int_{\Omega}\left(\frac{\lambda_k}{2} - \dfrac{\underline{v}(x)}{2}\right) u_0^2(x) \,dx < 0\,,
\]
which implies that
\begin{equation}\label{eq:GeometrySaddlePoint1}
\limsup_{u \in H_k,\, \|u\|_{\mathbb{X}}\to +\infty}\dfrac{\mathcal{J}(u)}{\|u\|^2_{\mathbb{X}}} <0\,.
\end{equation}
Moreover, once again mimicking the argument adopted in Case 1, we have that
\begin{equation}\label{eq:GeometrySaddlePoint2}
\liminf_{u \in \mathbb{P}_{k+1}, \, \|u\|_{\mathbb{X}}\to +\infty}\dfrac{\mathcal{J}(u)}{\|u\|^2_{\mathbb{X}}} >0\,.
\end{equation}

It follows from \eqref{eq:GeometrySaddlePoint2} that for any positive constant $M >0$, there exists a positive constant $R>0$ such that if $u \in \mathbb{P}_{k+1}$ with $\|u\|_{\mathbb{X}} \geq R$, then $\mathcal{J}(u)\geq M$. On the other hand, if $u \in \mathbb{P}_{k+1}$ with $\|u\|_{\mathbb{X}} \leq R$, a direct application of $\eqref{f_{lg}}$, \eqref{eq:Stima1}, H\"{o}lder's and Poincar\'e's inequalities gives
\begin{equation*}
\mathcal{J}(u) \geq \frac{\lambda_{k+1}}{2}\|u\|^2_{L^{2}(\Omega)} - \|a\|_{L^{2}(\Omega)}\|u\|_{L^{2}(\Omega)} - \dfrac{b}{2}\|u\|^2_{L^{2}(\Omega)} \geq -C\,,
\end{equation*}
where $C=C(R,\Omega,\|a\|_{L^2(\Omega)},b)>0$ is a positive constant. 
Therefore, we have that
\begin{equation*}
\mathcal{J}(u) \geq M - C, \quad \textrm{for every $u \in \mathbb{P}_{k+1}$}.
\end{equation*}
On the other hand, by \eqref{eq:GeometrySaddlePoint1}, we can choose a positive number $T>0$ such that
\begin{equation*}
\sup_{u\in H_k, \, \|u\|_{\mathbb{X}}=T}\mathcal{J}(u) < M - C\,.
\end{equation*}
It readily implies that
\begin{equation*}
\sup_{u\in H_k, \, \|u\|_{\mathbb{X}}=T}\mathcal{J}(u) < M - C \leq \inf_{u \in \mathbb{P}_{k+1}}\mathcal{J}(u)\,,
\end{equation*}
which in turn proves the validity of the topological requirements
to apply the Saddle Point Theorem.

It now remains to show the validity of the Palais-Smale condition.
Since the space $\mathbb{X}(\Omega)$ compactly embeds in $L^{2}(\Omega)$, it is enough to prove that the Palais-Smale sequences are bounded.  Arguing by contradiction, assume that the sequence $\{u_j\}_j$ is unbounded, define the normalized sequence $w_j := \tfrac{u_j}{\|u_j\|_{\mathbb{X}}}$ and assume the existence of $u_0 \in\mathbb{X}(\Omega)$ such that $w_j \to u_0$ (up to subsequences) as $j\to +\infty$, weakly in $\mathbb{X}(\Omega)$, strongly in $L^{2}(\Omega)$ and a.e. in $\Omega$.
Since $\{\mathcal{J}'(u_j)\}_j$ is bounded, we can infer the existence of a positive constant $M>0$ such that
\begin{equation}\label{eq:J'bdd}
\dfrac{|\mathcal{J}'(u_j)(\varphi)|}{\|u_j\|_{\mathbb{X}}} = \left| \mathcal{B}_{\alpha}(w_j,\varphi) - \int_{\Omega}\dfrac{f(x,u_j)}{\|u_j\|_{\mathbb{X}}}\varphi \, dx \right| \leq M \dfrac{\|\varphi\|_{\mathbb{X}}}{\|u_{j}\|_{\mathbb{X}}} \quad \textrm{ for every } \varphi \in \mathbb{X}(\Omega)\,.
\end{equation}
Recalling the standing assumption $\eqref{f_{lg}}$ on $f$, we have that
\begin{equation}\label{eq:StimaPerLG}
\dfrac{|f(x,u_j)|}{\|u_j\|_{\mathbb{X}}} \leq \dfrac{a(x)}{\|u_j\|_{\mathbb{X}}} + b \dfrac{|u_j|}{\|u_j\|_{\mathbb{X}}}\,,
\end{equation}
and we notice that the r.h.s. of \eqref{eq:StimaPerLG} is bounded in $L^{2}(\Omega)$. Therefore, there exists $\beta \in L^{2}(\Omega)$ such that (up to subsequences)
\[
\frac{f(x,u_j)}{\|u_j\|_{\mathbb{X}}}\quad\text{weakly converges to }\beta\text{ in }L^{2}(\Omega)\,,\text{ as }j\to+\infty\,.
\]
\medskip

\noindent\textbf{Claim A}: there exists a measurable function $m:\Omega \to \mathbb{R}$ such that
\begin{itemize}
\item[$(i)$] $\beta(x) = m(x) u_0(x)$ for a.e. $x \in \Omega$;
\item[$(ii)$] $\underline{v}(x) \leq m(x)\leq \overline{v}(x)$ for a.e. $x \in \Omega$.
\end{itemize}
\medskip

In order to prove the claim, we first notice that, if $x\in \Omega$ is such that $u_0(x)> 0$, then $u_j = w_j \|u_j\|_{\mathbb{X}}\to +\infty$ and then, recalling \eqref{eq:Defvv},
\begin{equation*}
\liminf_{j\to +\infty}\dfrac{f(x,u_j)}{\|u_j\|_{\mathbb{X}}}= \liminf_{j\to+\infty}\dfrac{f(x,u_j)}{u_j}\dfrac{u_j}{\|u_j\|_{\mathbb{X}}} \geq \underline{v}(x)u_0(x)\,,
\end{equation*}
while
\begin{equation*}
\limsup_{j\to +\infty}\dfrac{f(x,u_j)}{\|u_j\|_{\mathbb{X}}}= \limsup_{j\to+\infty}\dfrac{f(x,u_j)}{u_j}\dfrac{u_j}{\|u_j\|_{\mathbb{X}}} \leq \overline{v}(x)u_0(x)\,.
\end{equation*}
If $x\in \Omega$ is such that $u_0(x)<0$, both inequalities are reversed. 
Now, recall that, if $\{v_j\}_j$ weakly converges to $v$ in $L^2(\Omega)$, $g_j\leq v_j$ and $g_j$ converges to $g$ strongly in $L^2(\Omega)$ and a.e. in $\Omega$, then $g\leq v$ a.e. in $\Omega$. In this way, setting
\[
m(x) :=\begin{cases}
 \dfrac{\beta(x)}{u_0(x)}\,, & \textrm{ if } u_0(x)\neq 0\\
\dfrac{\underline{v}(x)+\overline{v}(x)}{2}\,, & \textrm{ if }  u_0(x) =0
\end{cases}
\]
we complete the proof of $(ii)$, and thus of Claim A, since, if $u_0(x)=0$, then from \eqref{eq:StimaPerLG} we have
\[
\dfrac{|f(x,u_j)|}{\|u_j\|_{\mathbb{X}}} \leq \dfrac{a(x)}{\|u_j\|_{\mathbb{X}}} + b \dfrac{|u_j|}{\|u_j\|_{\mathbb{X}}} \to 0\,, \quad \textrm{as } j\to +\infty\,.
\]

Then, by passing to the limit in \eqref{eq:J'bdd}, we get
\begin{equation}\label{b0}
\mathcal{B}_{\alpha}(u_0, \varphi) - \int_{\Omega}m(x)u_0(x)\varphi(x)\, dx = 0 \quad \textrm{for every } \varphi \in \mathbb{X}(\Omega)\,.
\end{equation}

\noindent\textbf{Claim B}: $u_0 =0$ a.e. in $\Omega$.\\
Since $u_0 \in \mathbb{X}(\Omega)$, we can write 
\[
u_0 = u_1 + u_2, \quad u_1 \in H_k, \, u_2 \in \mathbb{P}_{k+1}.
\]
Now, take $\varphi=u_1$ and then $\varphi=u_2$ in \eqref{b0}. Comparing, we get
\begin{equation}\label{b1}
\mathcal{B}_{\alpha}(u_1,u_1) - \int_{\Omega}m(x)u_1^2(x)\, dx = \mathcal{B}_{\alpha}(u_2,u_2) - \int_{\Omega}m(x)u_2^2(x)\, dx \,.
\end{equation}
Keeping in mind \eqref{eq:Stima1} and \eqref{eq:Stima2}, from \eqref{b1} we find
\[
\begin{aligned}
0 & \geq \int_{\Omega}(\lambda_k - m(x))u_1^2(x)\, dx \geq \mathcal{B}_{\alpha}(u_1,u_1) - \int_{\Omega}m(x)u_1^2(x)\, dx \\
&= \mathcal{B}_{\alpha}(u_2,u_2) - \int_{\Omega}m(x)u_2^2(x)\, dx 
\geq \int_{\Omega}(\lambda_{k+1}-m(x))u_2^2(x) \, dx \geq 0\,,
\end{aligned}
\]
but this is impossible unless $u_1 = u_2 =0$ a.e. in $\Omega$, and this proves Claim B.

We now test \eqref{eq:J'bdd} with
$\varphi = w_j$, getting
\begin{equation*}
\left|1 + \alpha \dfrac{[u_j]_s^2}{\|u_j\|^2_{\mathbb{X}}} - \int_{\Omega}\dfrac{f(x,u_j)}{u_{j}}\dfrac{u_j}{\|u_j\|_{\mathbb{X}}}\right| \leq \dfrac{M}{\|u_j\|_{\mathbb{X}}}\quad\text{for every }j\in\mathbb{N}\,.
\end{equation*}
Passing to the limit as $j\to+\infty$ we finally reach the contradiction ''1=0" and this closes the proof.
\end{proof}


\section{The superlinear and subcritical case}\label{sec.Superlinear}
In this section we prove Theorem \ref{thmSL}. We first need the following preliminary result inspired by Rabinowitz \cite{Rabinowitz78}.

\begin{lemma}\label{lemma_claim}
Let $k\in \N$ be such that
\[
\lambda_{k}\leq\underline{\Theta}(x)\leq\overline{\Theta}(x)<\lambda_{k+1}\quad\text{for a.e. }x\in\Omega\,,
\]
and decompose the space $\mathbb{X}(\Omega)$ as $\mathbb{X}(\Omega)=H_k \oplus \mathbb{P}_{k+1}$, where  $H_k :=\textrm{span}(u_1, \ldots, u_{k})$.
Then, there exists a positive constant $\beta$ such that
\begin{equation}\label{claim}
    \mathcal{B}_\alpha(u,u)-\int_\Omega \overline{\Theta}(x)u^2dx\geq\beta\|u\|_{\mathbb{X}(\Omega)}^2\quad\text{for any }u\in \mathbb{P}_{k+1}
\end{equation}
or, equivalently,
\[
\inf_{u\in \mathbb{P}_{k+1} \setminus\{0\}}\left\{1+\frac{\alpha[u]_s^2-\int_\Omega \overline{\Theta}(x)u^2dx}{\|u\|_{\mathbb{X}(\Omega)}^2}\right\}>0\,.
\]
\end{lemma}
\begin{proof}
By contradiction, assume the existence of a sequence $\{u_n\}_n\subset \mathbb{P}_{k+1} \setminus\{0\}$ satisfying
\[
1+\frac{\alpha[u_n]_s^2-\int_\Omega \overline{\Theta}(x)u_n^2dx}{\|u_n\|_{\mathbb{X}(\Omega)}^2}\leq 0\quad\text{for any }n\in\mathbb{N}
\]
and denote $v_n=\frac{u_n}{\|u_n\|_{\mathbb{X}(\Omega)}}$ for any $n\in\mathbb{N}$.
Then, the sequence $\{v_n\}_n$ is bounded in $\mathbb{P}_{k+1}$ and, by reflexivity, weakly convergent (up to subsequences) to $v$ in $\mathbb{P}_{k+1}$. Therefore, up to a further subsequence, $\{v_n\}_n$ strongly converges to $v$ in $L^2(\Omega)$ and in $H^s(\mathbb{R}^n)$, and so 
\begin{equation}\label{absurd}
 1+\alpha[v_n]_s^2-\int_\Omega \overline{\Theta}(x)v_n^2dx\rightarrow1+\alpha[v]_s^2-\int_\Omega \overline{\Theta}(x)v^2dx\leq 0.
\end{equation}

Therefore, by \eqref{eq:Stima1} and \eqref{absurd},
\[
0\leq \int_\Omega(\lambda_{k+1}-\overline{\Theta}(x))v^2dx\leq \mathcal{B}_\alpha(v,v)-\int_\Omega \overline{\Theta}(x)v^2dx\leq 0.
\]
Since $\overline{\Theta}<\lambda_{k+1}$ a.e. in $\Omega$, we get $v=0$ and so, by \eqref{absurd}
\[
1=1+\alpha[v]_s^2-\int_\Omega \overline{\Theta}(x)v^2dx\leq0\,,
\]
which yields a contradiction.
\end{proof}
\begin{proof}[Proof of Theorem $\ref{thmSL}$]
We decompose $\mathbb{X}(\Omega)=H_k \oplus \mathbb{P}_{k+1}$,  and we show the existence of $\varrho,\tilde\alpha>0$ such that
\[
\inf_{u\in S_\varrho\cap \mathbb{P}_{k+1}}\mathcal{J}(u)=\tilde\alpha\,.
\]

We first we claim that for any $\varepsilon>0$ there exists $C_\varepsilon>0$ such that
\begin{equation}\label{eq.Claim2}
F(x,t)\leq\frac{1}{2}(\overline{\Theta}(x)+\varepsilon)t^2+C_\varepsilon|t|^r\quad\text{a.e. }x\in\Omega\text{ for all }t\in\mathbb{R}\,.
\end{equation}
Indeed, if we fix $\varepsilon>0$, by \eqref{iv}, there exists $\delta=\delta(\varepsilon)>0$ such that
\begin{equation}\label{delta1}
F(x,t)\leq\frac{1}{2}(\overline{\Theta}(x)+\varepsilon)t^2\quad\text{a.e. }x\in\Omega\text{ for all }|t|\leq\delta\,.
\end{equation}
Moreover, by \eqref{(ii)}
\begin{equation}\label{delta2}
F(x,t) \leq|\int_0^t(a(x)+b|\sigma|^{r-1})\,d\sigma|\leq a(x)|t|+b\frac{|t|^r}{r}\quad\text{a.e. }x\in\Omega\text{ for all }t\in\mathbb{R}\,.
\end{equation}
Combining \eqref{delta1} (for $|t|\leq\delta$) and \eqref{delta2} (for $|t|\geq\delta$), we finally get
\begin{align*}
    F(x,t)&\leq\frac{1}{2}(\overline{\Theta}(x)+\varepsilon)t^2+a(x)|t|\frac{|t|^{r-1}}{|t|^{r-1}}+\frac{b}{r}|t|^r\\
    &\leq\frac{1}{2}(\overline{\Theta}(x)+\varepsilon)t^2+(\frac{a(x)}{\delta^{r-1}}+\frac{b}{r})|t|^r\\
    &\leq \frac{1}{2}(\overline{\Theta}(x)+\varepsilon)t^2+C_\varepsilon|t|^r\quad\text{a.e. }x\in\Omega\text{ for all }t\in\mathbb{R}\,.
\end{align*}
Notice that, by \eqref{eq.Claim2} and the Sobolev inequality, we also have
\begin{equation}\label{StimesuF}
\begin{split}
\int_\Omega F(x,u)\,dx&\leq\frac{1}{2}\int_\Omega(\overline{\Theta}(x)+\varepsilon)|u|^2\,dx+C_\varepsilon\int_\Omega |u|^r\,dx\\
&\leq\frac{1}{2}\int_\Omega(\overline{\Theta}(x)+\varepsilon)|u|^2\,dx+\overline{C_\varepsilon}\|u\|_{\mathbb{X}}^r
\end{split}
\end{equation}
for any $u\in\mathbb{X}(\Omega)$.

Now, take $u\in \mathbb{P}_{k+1}$. Then, by \eqref{claim}, \eqref{StimesuF} and the Poincar\'e inequality
\begin{align*}
    \mathcal{J}(u)&=\frac{1}{2}\mathcal{B}_\alpha(u,u)-\int_\Omega F(x,u)\,dx\\
    &\geq\frac{1}{2}\mathcal{B}_\alpha(u,u)-\frac{1}{2}\int_\Omega(\overline{\Theta}(x)+\varepsilon)|u|^2\,dx-\overline{C_\varepsilon}\|u\|_{\mathbb{X}}^{r}\\
    &\geq\frac{1}{2}\beta\|u\|_{\mathbb{X}}^2-\frac{\varepsilon}{2}\|u\|_{L^2(\Omega)}^2-\overline{C_\varepsilon}\|u\|_{\mathbb{X}}^{r}\\
    &\geq\left(\frac{1}{2}\beta-\frac{\varepsilon}{2\mu_1} - \overline{C_\varepsilon}\|u\|_{\mathbb{X}}^{r-2}\right)\|u\|_{\mathbb{X}}^2\,,
\end{align*}
where $\mu_1$ denotes the first eigenvalue of $-\Delta$ in $\Omega$. Now we choose $\varepsilon$ and $\varrho>0$ so small that
\[
\inf_{u\in S_\varrho\cap \mathbb{P}_{k+1}}\mathcal{J}(u)=\tilde\alpha>0\,.
\]
To verify that $\mathcal{J}$ satisfies the geometric condition of the Linking theorem, we show the existence of a radius $\rho>\varrho$ such that
\[
\sup_{u\in\partial_{H_k \oplus\textrm{span}(u_{k+1})}\Delta}\mathcal{J}(u)\leq0\,,
\]
where $\Delta:=(\overline{B}_\rho\cap H_k)\oplus\{tu_{k+1}:t\in[0,\rho]\}$.

For this, first let us notice that $\mathcal{J}(u)\leq0$ for any $u\in  H_k$, as we easily get by combining \eqref{eq:Stima2}  with the assumption \eqref{lastassumption}.
Therefore,
\[
\sup_{ H_k}\mathcal{J}= 0\,.
\]
Moreover, notice that there exists a positive constant $c_1$ (if $\alpha<0$, then $c_1=1$) such that
\begin{equation}\label{equinorms}
    \mathcal{B}_\alpha(u,u)\leq c_1\|u\|_{\mathbb{X}}^2\quad\text{for any }u\in\mathbb{X}(\Omega).
\end{equation}

Now take $t>0$ and $u\in X_1$. Then, by \eqref{(iii)} and \eqref{equinorms} we have
\begin{equation*}
\begin{split}
    \mathcal{J}(u+tu_{k+1})&=\frac{1}{2}\mathcal{B}_\alpha(u+tu_{k+1},u+tu_{k+1})-\int_\Omega F(x,u+tu_{k+1})\,dx\\
    &\leq\frac{c_1t^2}{2}\|\frac{u}{t}+u_{k+1}\|_\mathbb{X}^2-ct^{\tilde \mu}\|\frac{u}{t}+u_{k+1}\|_{L^{\tilde\mu}(\Omega)}^{\tilde \mu}+\|d\|_{L^1(\Omega)}\to-\infty
\end{split}
\end{equation*}
as $t\to\infty$, having assumed $\tilde \mu>2$.

In this way the geometric assumptions of the Linking Theorem are satisfied. Notice that, letting $k=0$, we recover the geometric situation of the Mountain Pass Theorem when $\overline{\Theta}(x)<\lambda_1$ a.e. $x \in \Omega$.

We conclude by showing the validity of the Palais-Smale condition, that is, that every Palais-Smale sequence, i.e. every sequence $\{u_j\}_j$ in $\X$ such that  $\{\mathcal{J}(u_j)\}_j$ is bounded and $\mathcal{J}'(u_j)\to 0$ in $\X^{-1}$ as $j\to\infty$, has a strongly converging subsequence in $\mathbb{X}(\Omega)$.


So, let $\{u_j\}_j\subset\mathbb{X}(\Omega)$ be a Palais-Smale sequence. We first prove that $\{u_j\}_j$ is bounded in $\mathbb{X}(\Omega)$.

Assume, by contradiction, that $\{u_j\}_j$ is unbounded in $\mathbb{X}(\Omega)$. Then, there exists $u\in\mathbb{X}(\Omega)$ such that (up to subsequences)
\[
\frac{u_j}{\|u_j\|_\mathbb{X}}\to u\quad\text{weakly in }\mathbb{X}(\Omega) \mbox{ and strongly  in $H^s_0(\Omega)$, in $L^{\tilde \mu}(\Omega)$ and a.e. in $\Omega$} .
\]
Now, take $\eta\in (2,\mu)$. Then, by $(ii)$ and $(iii)$, there exists a positive constant $D_R$ such that
\begin{equation}\label{primostep}
\begin{aligned}
    0\leftarrow\frac{\eta\mathcal{J}(u_j)-\mathcal{J}'(u_j)u_j}{\|u_j\|_X^2}&=\left(\frac{\eta}{2}-1\right)\frac{\mathcal{B}_\alpha(u_j,u_j)}{\|u_j\|_\mathbb{X}^2}\\
    &\quad+\frac{\int_{\{|u_j|> R\}}(f(x,u_j)u_j\pm A(x)u_j^2-\eta F(x,u_j))dx}{\|u_j\|_\mathbb{X}^2}\\
    &\quad-\frac{\int_{\{|u_j|\leq R\}}(\eta F(x,u_j)-f(x,u_j)u_j)dx}{\|u_j\|_\mathbb{X}^2}\\
    &\geq \left(\frac{\eta}{2}-1\right)\left(1+\alpha\frac{[u_j]_s^2}{\|u_j\|_\mathbb{X}^2}\right)+\dfrac{(\mu-\eta)}{\|u_j\|_\mathbb{X}^2}\int_\Omega F(x,u_j)\,dx\\
    &\quad+\left(1-\frac{\mu}{2}\right)\frac{\int_\Omega Au_j^2dx}{\|u_j\|_\mathbb{X}^2}-\frac{D_R}{\|u_j\|_\mathbb{X}^2}\,.
\end{aligned}
\end{equation}
By Young's inequality, for every $\varepsilon>0$, there exists a constant $\tilde D>0$ such that
\begin{equation}\label{step2}
\left|\left(1-\frac{\mu}{2}\right)\right|\|A\|_{L^\infty(\Omega)}\|u_j\|^2_{L^2(\Omega)}\leq \varepsilon \int_\Omega|u_j|^{\tilde \mu}dx+\tilde D.
\end{equation}
By \eqref{primostep}, \eqref{step2} and $(iii)$ we get
\begin{equation}\label{giustosopra}
\begin{split}
    0\leftarrow \frac{\eta\mathcal{J}(u_j)-\mathcal{J}'(u_j)u_j}{\|u_j\|_\mathbb{X}^2}
    &\geq\left(\frac{\eta}{2}-1\right)\left(1+\alpha\frac{[u_j]_s^2}{\|u_j\|_\mathbb{X}^2}\right)\\
    &\quad+[(\mu-\eta)c-\varepsilon]\int_\Omega \frac{|u_j|^{\tilde \mu}}{\|u_j\|_\mathbb{X}^2}dx+o(1)\,,
\end{split}
\end{equation}
where $o(1)\to 0$ as $j\to \infty$. Choosing $\varepsilon <(\mu-\eta)c$, we immediately get that
\[
1+\alpha[u]_s^2\leq 0, 
\]
so that $u\neq 0$ if $\alpha<0$, while it is already a contradiction if $\alpha\geq0$. On the other hand, if in \eqref{giustosopra} we divide by $\|u_j\|_\mathbb{X}^{\tilde\mu-2}$ and pass to the limit, we get $u=0$, which yields a contradiction.

Thus, $\{u_j\}_j\subset\mathbb{X}(\Omega)$ is bounded in $\mathbb{X}(\Omega)$ and, by reflexivity and Rellich's theorem, there exists $u\in\mathbb{X}(\Omega)$ such that (up to subsequences)
\begin{center}
$u_j\to u$ weakly in $\mathbb{X}(\Omega)$ and strongly in $H^s_0(\Omega)$, in $L^p(\Omega)$ $(p<2^*)$ and a.e. in $\Omega$.
\end{center}
Now, it is standard to prove that $u_j\to u$ strongly in $\X$.
\medskip

Hence the Mountain Pass Theorem (if $\overline{\Theta}(x)<\lambda_1$) or the Linking Theorem can  be applied and Theorem \ref{thmSL} holds.
\end{proof}

\section*{Statements and Declarations}
\noindent\textbf{Funding} A. Maione is supported by the DFG SPP 2256 project ``Variational Methods for Predicting Complex Phenomena in Engineering Structures and Materials'', by the University of Freiburg and by INdAM-GNAMPA Project ``Equazioni differenziali alle derivate parziali in fenomeni non lineari''. D. Mugnai is supported by the FFABR ``Fondo per il finanziamento delle attivit\`a base di ricerca'' 2017 and by the INdAM-GNAMPA Project ``PDE ellittiche a diffusione mista''.
E. Vecchi is supported by INdAM-GNAMPA Project ``PDE ellittiche a diffusione mista'' (INdAM-GNAMPA Projects Grant number CUP\_E55F22000270001).

\noindent\textbf{Acknowledgements.} The authors wish to thank Enzo Vitillaro and Maicol Caponi for useful discussions,  and the anonymous referees for their careful reading. All their comments  led to an improvement of the first version of the paper.

\noindent\textbf{Conflict of interest} The authors declare that they have no conflict of interest.

\end{document}